\documentclass{article}
\usepackage{amsfonts}
\usepackage{amsmath}
\usepackage{amssymb}
\usepackage{amsthm}
\newtheorem{thm}{Theorem}[subsection]

\newtheorem{lem}[thm]{Lemma}
\newtheorem{prop}[thm]{Proposition}
\newtheorem{cor}[thm]{Corollary}
\newtheorem{defn}[thm]{Definition}
\newtheorem{q}[thm]{Question}
\newtheorem{rem}[thm]{Remark}
\newcommand\numberthis{\addtocounter{equation}{1}\tag{\theequation}}
\begin{document}

\title{A Negative Answer to a Problem of Aldous on Determination of Exchangeable Sequences}
\author{Jeff Lin\\UCLA}
\date{\today}
\maketitle

\begin{abstract}
We present results concerning when the joint distribution of an exchangeable sequence is determined by the marginal distributions of its partial sums.  The question of whether or not this determination occurs was posed by David Aldous.  We then consider related uniqueness problems, including a continuous time analog to the Aldous problem and a randomized univariate moment problem.

\end{abstract}
\newpage
\section{Introduction}

In probability theory, there are many results concerning uniqueness of a distribution satisfying certain properties.  For instance, there are the various moment problems, the inversion of the characteristic function, the inversion of the Laplace transform.  Another kind of uniqueness result relates to exchangeable sequences of random variables.  Given the joint distribution of an exchangeable vector $X=(X_n)_{n \in \mathbb{N}}$, it can be written as a ``mixture'' of iids in exactly one way. 

We will prove altered versions of these types of results.  Roughly speaking, we will assume only partial information, and make regularity assumptions to ensure that the resulting problem is well-defined.  In the process, we show a trend of what types of obstructions there can be to such uniqueness results: arithmetic and algebraic structure. (These two notions are related because arithmetic relationships between exponents, say in a Laplace transform, will correspond to polynomial relationships between the exponentials, for instance the Laplace transforms.)

In \cite{A}, (p. 20) Aldous proposes the following question.

\begin{q}
Let $X=(X_j)_{j\in \mathbb{N}}$, $Y=(Y_j)_{j\in \mathbb{N}}$ be exchangeable sequences of $\mathbb{R}$-valued random variables, with $S_n=\sum_{j=1}^n X_j, T_n=\sum_{j=1}^n Y_j$.  Suppose that $\forall n \ge 1$ we have $S_n=_dT_n$.  Does it follow that $X=_dY$  \label{q1}?
\end{q}
Despite this question being over 30 years old, only partial progress has been made on it.  In \cite{EZ} Section 2, Evans and Zhou show the answer is negative in the class of signed random variables.  Therefore, interest has somewhat shifted to the nonnegative case.  In \cite{EZ}, it is shown that there is an affirmative answer if we know that $X$ and $Y$ are mixtures of countably many nonnegative iid sequences.  We will show that the nonnegative restriction does not improve the situation in the absence of the additional assumption of \cite{EZ}: $X$ cannot be determined uniquely from this partial information even if $X_j\ge0$. However, as we will see, uniqueness holds when the exchangeable sequences are mixtures of up to $3$ iid sequences.

Uniqueness is also true for more complicated mixtures, so long as the iid sequences involved in the mixture are related ``transcendentally''. (It is the main purpose of this paper to present results which make this theme precise, in the context of all uniqueness questions we will consider.)  We may think of the heuristic, ``mixtures of a small number of distributions exhibit uniqueness of $X$'', as a special case of the heuristic about arithmetic and algebraic (polynomial) dependence.  For instance, if the collection of allowed distributions in the is small, then there cannot be too many arithmetic dependences.

The first question naturally leads to the following continuous time analog, suggested by J. \v Cern\'y.

\begin{q}
Let $S_t$ and $T_t$ be mixtures of L\'evy Processes.  Suppose that $\forall t \ge0,\, S_t=_dT_t$.  Is $(S_t)_{t \ge0 }$ jointly equidistributed with $(T_t)_{t \ge0 }$\label{q2}?
\end{q}

We will see that in the case of restricting to mixtures of Brownian Motions, the answer is yes.  However, we will find a (possibly unexpected) parallel between a measure similar to the Levy measure implied in this problem and the mixing measure implied in Question \ref{q1}.  In particular, we will find that upon restricting to even just Poisson Processes, uniqueness fails.

For a L\'evy Process, knowing the marginal at any time $t\neq0$, such as $t=1$, tells us the entire joint distribution of the process.  However, this will not be true for a mixture of L\'evy Processes.  We will see that in some cases, we could say $S_t$ is unique just from making observations at $t\in \mathbb{N}$ and other times we require all $t\ge0$.  In the former case, the continuous time problem makes uniqueness of $S_t$ more plausible than the discrete time case only because continuous time imparts infinite divisibility which limits the possible distributions that can occur in the mixture.  In the latter case, the continuous time problems will tend to exhibit uniqueness over their discrete time counterparts not only because of infinite divisibility, but because observation at noninteger $t$'s provide more information beyond that.

An example of where the fact that we make continuum observations makes a difference, at least in the argument, is the case of a mixture of normal distributions or Brownian motions.  In this case, the continuum of observations destroys the arithmetic structure, making uniqueness a fact in the Brownian Motion case, as opposed to an open problem in the discrete time case.

In light of these questions, it is natural also to consider various random analogs of uniqueness problems from the classical theory of random variables.  For instance, one such a question are posed below:

\begin{q}
Let $\alpha, \beta$ be random Borel measures on $\mathbb{R}$.  Suppose that their moments (which are random variables) $\int_\mathbb{R} x^n d\alpha$ and $\int_\mathbb{R} x^n d\beta$ are well-defined and equal in marginal distribution (i.e. for one $n$ at a time).  Do we have $\alpha=_d \beta$\label{q3}?  
\end{q}


Question \ref{q3} can be viewed as a randomization of classical uniqueness theorems, in the sense that the probability measure one tries to recover is deterministic in the classical theories, and we now take it to be random.  

We will see that in all randomized versions of classical uniqueness problems we consider, knowing joint information is enough to assert uniqueness of the original random measure in distribution, but knowing marginal information is not enough.  We will see that the existence of counterexamples again depends on the arithmetic and algebraic structure.

The layout of the paper is as follows.  We will treat Question \ref{q1} in Section 2, Question \ref{q2} in Section 3, and the questions similar to Question \ref{q3} in Section 4.  Section 3 is fairly technical, but Section 4 does not depend on Section 3, so Section 3 could be omitted on a first pass.  In each of the sections, we will first discuss the results required to make the questions precise and the definitions associated to the corresponding question.  We also discuss the machinery to be used in the rest of the section. In subsections, we consider answers to the question in various cases.

\section{Discrete Time Exchangeability Problem}

In this section, we consider Question \ref{q1} posed by Aldous.  We keep the notation used in the statement of the question, so that $X=(X_j)$ and $Y=(Y_j)$ will denote sequences of exchangeable random variables.

First we state the following without proof:

\begin{lem}[Classical Bounded Moment Problem] Let $V=(V_\gamma)$ and $Z=(Z_\gamma)$ be vectors indexed by the same set $\Gamma$ of arbitrary cardinality.  Thus $V$ and $Z$ are valued in $\mathbb{R}^\Gamma$.  Assume that for each $\gamma \in \Gamma$, $V_\gamma, Z_\gamma$ are bounded real random variables.  Then $V$ and $Z$ have the same distribution if and only if they have the same joint moments, i.e. $\forall \{\gamma_1, \dots, \gamma_l\} \subset \Gamma$ $\forall r_1, \dots, r_l \in \mathbb{N}$ 
\begin{equation}\mathbb{E}\left(\prod_{i=1}^l V_{\gamma_i}^{r_i} \right)=\mathbb{E}\left(\prod_{i=1}^l Z_{\gamma_i}^{r_i} \right)\end{equation}.\label{arbmoment}
\end{lem}

We will have need throughout the paper for the notion of a random measure.  In general, we will use $P(\mathfrak{M})$ to denote the (Borel) probability measures defined on a standard Borel space $\mathfrak{M}$.  Similarly, we define $M_+(\mathfrak{M})$ to be the collection of finite nonnegative Borel measures on $\mathfrak{M}$.  In each case, give these spaces the measurable structure generated by the mapping $\mu \mapsto \mu(B), B\in \text{Borel}(\mathfrak{M})$. Since $\mathfrak{M}$ is standard Borel, so is $P(\mathfrak{M})$ with the topology given by vague convergence.  The Borel $\sigma$ algebra determined by vague convergence is an equivalent definition of the measurable structure given to $P(\mathfrak{M})$. (For these facts, see for instance \cite{K}.)  A random measure is then a random variable taking values in either $M_+(\mathfrak{M})$ or $P(\mathfrak{M})$.  We will always deal with Polish spaces when the topology is significant, and standard Borel spaces if not. We almost always deal with probability measures, i.e. except when dealing with L\'evy measures.

From now on we will use 
\begin{equation}\forall s\ge0,\, \mathcal{L}_\mu(s) =\int_{[0,\infty)} e^{-sx}d\mu(x)\end{equation} 
to mean the Laplace transform (at $s\ge0$) of a probability measure $\mu$ on $[0,\infty)$.  Let $P^+=\{\mu \in P(\mathbb{R})|\mu \text{ is supported in } [0,\infty)\}$, which is closed in $P(\mathbb{R})$.  If $\mu$ is random, then the Laplace transform will simply be a random variable. 

In this section, we will use symbols $\alpha, \alpha_1, \alpha_2$ to denote random probability measures and $\Theta, \Theta_1, \Theta_2$ to denote elements of $P(P(\mathbb{R}))$.  Symbols $X_j, Y_j$ will denote real-valued random variables.

\begin{defn}

For each $j\ge1$, let $X_j$ be a real-valued random variable.  Call $X=(X_j)_{j\in\mathbb{N}}$ \textbf{exchangeable} if whenever $\pi:\mathbb{N}\rightarrow \mathbb{N}$ is a finite permutation, we have that 
\begin{equation}(X_j)_{j\in \mathbb{N}}=_d(X_{\pi(j)})_{j\in \mathbb{N}}.\end{equation}
\end{defn}

\begin{defn}
Say that $X$ is \textbf{iid-$\Theta$} if its distribution is given by
\begin{equation}\forall A \in Borel(\mathbb{R}^\mathbb{N}),\, P(X \in A)=\int{\theta^\mathbb{N}(A)\,d\Theta(\theta)}\label{iid-theta}\end{equation}.
\end{defn}

\begin{defn}

Say that $X$ is a \textbf{mixture of iids directed by $\alpha$}, where $\alpha$ is defined on the same probability space as $X$, if $\alpha^\mathbb{N}$ is a regular conditional distribution for $X$ given $\sigma(\alpha)$. 

\end{defn}

By definition of regular conditional distribution this is just the same as saying that whenever $A_1, \dots, A_n \in Borel(\mathbb{R})$ we have for a probability $1$ set of $\omega$,
\begin{equation}P(X_i \in A_i, 1 \le i \le n|\sigma(\alpha))(\omega)= \prod_i \alpha(\omega, A_i).\end{equation}
We now state the well-known theorem of De Finetti.

\begin{thm}[De Finetti] The following are equivalent:
\newline 1. X is exchangeable
\newline 2. X is a mixture of iids directed by some $\alpha$
\newline 3. X is iid-$\Theta$ for some $\Theta$.
\end{thm}

\begin{rem} For each exchangeable law $\mu$ on $\mathbb{R}^\mathbb{N}$ there exists a unique law $\Theta$ on $P(\mathbb{R})$ for which any $X$ with distribution $\mu$ is iid-$\Theta$.  This assignment $\mu \mapsto \Theta$ is a bijection from the exchangeable laws on $\mathbb{R}^\mathbb{N}$ to $P(P(\mathbb{R}))$.  In fact, it is a homeomorphism.  Also, if $X$ is exchangeable then its directing measure is unique up to a.s. equality, and the distribution of the RV $\alpha$ is $\Theta$.  Therefore, if the (joint) distribution of $X$ is determined, then the law of $\alpha$ is determined. \label{corresp}
\end{rem}

In the literature, the phrase ``mixing measure'' refers to both $\Theta$ and $\alpha$.  We will use the terminology for $\Theta$ and call $\alpha$ the directing measure.  In contexts when we have two exchangeable sequences $X$ and $Y$, we denote by $\Theta_1, \alpha_1$ the mixing measure, directing measure (respectively) for $X$ and $\Theta_2, \alpha_2$ the mixing measure, directing measure (respectively) for $Y$.  Symbols $\alpha_1, \alpha_2, \alpha$ will always refer directing measures, not arbitrary random measures.

We set $S_n=\sum_{j=1}^n X_j$ and $T_n=\sum_{j=1}^nY_j$.  Our question is then whether $\forall n,\, S_n=_dT_n$ implies $\Theta_1=\Theta_2$ (or $\alpha_1=_d \alpha_2$).

Towards this, let us define:

\begin{defn}
Let $L\subset P(\mathbb{R})$ be measurable.  We will say that $L$ is \textbf{good} provided that whenever $\Theta_1, \Theta_2$ are concentrated on $L$ and $\forall n,\, S_n=_dT_n$, we have $\Theta_1=\Theta_2$.
\end{defn}

To determine if a class $L$ is good or not, we may use the following

\begin{lem}
A measurable class $L\subset P^+$ of distributions is good if and only if $\forall \alpha_1, \alpha_2$ a.s. $L$-valued, we have that 
\begin{equation}\forall s\ge0,\, \mathcal{L}_{\alpha_1}(s)=_d\mathcal{L}_{\alpha_2}(s)\end{equation}
implies
$\Theta_1=\Theta_2$.
\label{equiv}
\end{lem}






\begin{proof}

We have that the statement $\forall n,\, S_n=_dT_n$ is equivalent to the statement

\begin{equation}\forall n, s\ge0,\, \mathbb{E}[e^{-sS_n}]=\mathbb{E}[e^{-sT_n}],\end{equation}
which, using the definition of a directing measure, is equivalent to

\begin{equation}\forall n, s\ge0,\, \mathbb{E}\left[\left(\int_{[0, \infty)}e^{-sx}d\alpha_1(x)\right)^n\right]=\mathbb{E}\left[\left(\int_{[0, \infty)}e^{-sx}d\alpha_2(x)\right)^n\right].\end{equation}
By the classical bounded moment problem, this is equivalent to 
\begin{equation}\forall s \ge0,\, \mathcal{L}_{\alpha_1}(s)=_d \mathcal{L}_{\alpha_2}(s).\end{equation}
\end{proof}

We do not assert the analogous statement for MGFs (moment generating functions) because we have no need for it even though it is true under suitable boundedness hypotheses. We do not assert a statement for the characteristic function case because the characteristic function would be a complex-valued random variable, and one would need the joint distribution of the real and imaginary parts. The absence of an analogous lemma for characteristic functions is related to the absence of a bounded moment problem for complex-valued RVs that does not involve knowing any conjugate moments. This can be thought of as a heuristic reason for the shift in focus away from signed random variables in the recovery problems we consider in this article.

\subsection{Known Results}

In this section, we will discuss results known beforehand.  Using the conditional SLLN, it follows that

\begin{prop}
$P(\{0, 1\})$ is good.
\end{prop}

In Section 2 of \cite{EZ}, it is shown that

\begin{prop}
$P(\mathbb{R})$ is not good.
\end{prop}

In Section 3 of the same paper, it is shown that a positive result can still be salvaged.

\begin{prop}
If $\forall n,\, S_n=_dT_n$, and $\Theta_1, \Theta_2$ are purely atomic, concentrated in $P^+$ then $X=_dY$. \label{ptmass}  In other words, any countable subset of $P^+$ is good.
\end{prop}

Interestingly enough, Muntz's Theorem, which arises in Lemma 3.3 of \cite{EZ}, can be replaced by the theorem of complex analysis stating that holomorphic functions defined on a connected open set agree just as soon as they agree on a set of points that accumulates within the domain.  

Their assumption is of a different nature than what we will consider; they restrict the size of the sets where $\Theta_1, \Theta_2$ are concentrated rather than requiring, as we will, that $\Theta_1, \Theta_2$ are concentrated on a nice set of distributions that may be a continuum.  Considering that the general problem has been solved in the negative, we will primarily address restrictions of the latter type, hence the terminology of such a set $L$ being ``good''.  A particularly interesting case is when $L$ is contained in $P^+$.

For example, let $\mathcal{P}$ denote the collection of distributions in $P(\mathbb{R})$ which are the distribution probability measures for Poisson RVs. (So $\mu \in \mathcal{P}$ should satisfy $\forall k \ge0,\, \mu(\{k\})=e^{-\lambda}\lambda^k/k!$ for some $\lambda\ge0$.)  Similarly, define $\mathcal{E}$ for exponentials, $\mathcal{G}$ for geometrics, and $\mathcal{B}_n$ for binomials with parameter $n$ fixed and $p \in[0, 1]$ varying.  It is an exercise in analysis to see that:

\begin{rem}
$\mathcal{B}_n, \mathcal{G}, \mathcal{E}, \mathcal{P}$ are all good. \label{BGEP1}
\end{rem}

\subsection{The Relationship Between the Aldous Question and Arithmetic Properties of the Value Set}

In this subsection, we will present one substantiation of the heuristic that the more arithmetic dependences there are in the allowed value set for $X_1$, the harder it is to recover $X$ uniquely.  Fix $A=\{a_1, \dots, a_N \}$ when $N \in \mathbb{N}$ or $A=\{a_1, \dots, \}$ when $N=\infty$.  Assume, for technical reasons, that $A$ has no accumulation points in $\mathbb{R}$.  We allow negative values in $A$.  Define $F_A\subset P(\mathbb{R})$ to be those distributions that are supported in $A$. (We can say ``supported'' instead of ``concentrated'', because $A$ is closed.  $F_A$ is closed because $A$ is.) We set $L:=F_A$.   For $X$ exchangeable, it is equivalent to say that $X_1\in A$ a.s. or to say $\alpha \in L$ a.s. or to say $\Theta$ is supported on $L$.

\begin{thm}
$F_A$ is good if $A$ is linearly independent over $\mathbb{Q}$. \label{arithmetic independence}
\end{thm}

\begin{proof}
It suffices to handle the case in which $A$ is discrete, infinite, i.e. $N=\infty$, because subsets of good sets are always good.  We have $X, Y$ exchangeable sequences, which therefore comes with directing measures and mixing measures by Remark \ref{corresp}.  Since we are trying to prove goodness, we assume $\forall n \ge0,\, S_n=_dT_n$. (In future proofs of goodness, we not explicitly mention this.)  We need to see that $\alpha_1=_d \alpha_2$.  It suffices to see that $(\alpha_1(\{a_j\}))_{j\in \mathbb{N}}=_d(\alpha_2(\{a_j\}))_{j\in \mathbb{N}}$

Suppose that $M\in\mathbb{N}$.  Then by using the linear independence hypothesis for our value set and the properties of directing measures, we have that $\forall r_1, \dots, r_M$, with $s={\sum_{j=1}^M r_i}$

\begin{multline*}\binom{s}{r_1,r_2,\dotsc, r_M}\mathbb{E}(\prod_{i=1}^M (\alpha_1\{a_i\})^{r_i})=\mathbb{P}(S_s=\sum_{i=1}^s r_ia_i)=\\
=\mathbb{P}(T_s=\sum_{i=1}^s r_ia_i)=\binom{s}{r_1, r_2,\dotsc, r_M}\mathbb{E}(\prod_{i=1}^M (\alpha_2\{a_i\})^{r_i}).\numberthis\end{multline*}
Here, the first and third equality hold because the linear independence hypothesis guarantees that the only way that a sum of elements from $A$ can be $\sum_{i=1}^s r_ia_i$ is if $r_i$ copies of $a_i$ are used for each $i\le M$.

Now lemma \ref{arbmoment} completes the proof.
\end{proof}

It is also acceptable to have a few arithmetic dependences:

\begin{prop}
 $F_{\{0, 1, 2\}}$ is good. More generally, let $\mu_1, \mu_2, \mu_3 \in P^+$ and let $\mathcal{L}_1, \mathcal{L}_2, \mathcal{L}_3$ denote their Laplace transforms respectively.  Let 
\begin{equation}L:=\{a_1\mu_1+a_2\mu_2+a_3\mu_3|a_1+a_2+a_3=1, a_1, a_2, a_3\ge0\}.\end{equation}
Then $L$ is good.\label{3 value}
 
\end{prop}

\begin{proof}
We prove the more general claim.  We may assume that no strict subset of $\{\mu_1, \mu_2, \mu_3\}$ has convex hull containing all of $\mu_1, \mu_2, \mu_3$ because subsets of good sets are good.  From this, it follows that $L$ is homeomorphic to $T_3:=\{(a, b, c)|a+b+c=1, a, b, c\ge0\}$ because the map taking $(a, b, c) \mapsto a\mu_1+b\mu_2+c\mu_3$ is invertible, and therefore because $T_3$ and $L$ are both compact Hausdorff, they are homeomorphic via this assignment, hence measurably isomorphic.

Proceeding by way of Lemma \ref{equiv}, and using the homeomorphism above, we are reduced to showing that if $U, V$ are $T_3$-valued random vectors for which $\forall s\ge0$ we have 
\begin{equation}(\mathcal{L}_1(s), \mathcal{L}_2(s), \mathcal{L}_3(s))\cdot U=_d(\mathcal{L}_1(s), \mathcal{L}_2(s), \mathcal{L}_3(s))\cdot V\end{equation}
then $U=_d V$.  That this is enough follows from the fact that when $U, V$ are the pushforwards of $\alpha_1, \alpha_2$ via the homeomorphism above, the left side has the distribution of the (random) Laplace transform of $\alpha_1$ evaluated at $s$, and the right side has the distribution of the (random) Laplace transform of $\alpha_2$ evaluated at $s$.

Because $U, V$ are probability vectors, it suffices to show that if $U', V'$ are bounded random vectors in $\mathbb{R}^2$ with
\begin{equation}(\mathcal{L}_1(s)-\mathcal{L}_3(s), \mathcal{L}_2(s)-\mathcal{L}_3(s))\cdot U'=_d(\mathcal{L}_1(s)-\mathcal{L}_3(s), \mathcal{L}_2(s)-\mathcal{L}_3(s))\cdot V'\end{equation}
then $U'=_dV'$.

We will check that the collection of $v$'s of the form
\begin{equation}\label{form}v=c(\mathcal{L}_1(s)-\mathcal{L}_3(s), \mathcal{L}_2(s)-\mathcal{L}_3(s))\end{equation}
cover a nonempty open set as $s\ge0$ and $c\in \mathbb{R}$ vary.  We know that $\mu_1-\mu_3$ is a signed measure on $[0, \infty)$ and as is $\mu_2-\mu_3$.  We may regard $\mathcal{L}_1(s)-\mathcal{L}_3(s)$ and $\mathcal{L}_2(s)-\mathcal{L}_3(s)$ as the Laplace transforms of these signed measures respectively.  Therefore, a nonempty open set of $v$'s are of the form \eqref{form}.  For instance, one may check that the derivatives of $\mathcal{L}_1(s)-\mathcal{L}_3(s)$ and $\mathcal{L}_2(s)-\mathcal{L}_3(s)$ must be different at some $s_0>0$, for if they were the same, then $\mu_1-\mu_3$ and $\mu_2-\mu_3$ would have the same Laplace transform.

Bivariate analytic functions agreeing on a nonempty open set in $\mathbb{R}^2$ have to agree, and we apply this to the bivariate MGF of $U'$ and that of $V'$.
\end{proof}





\subsection{Four Values are Too Many}

We will see our first counterexample to the nonnegative Aldous problem, proving the claim

\begin{prop}
$P^+$ is not good.
\end{prop}

We will even show that

\begin{thm}
$F_{\{0, 1, 2, 3\}}$ is not good. \label{0123}
\end{thm}

In order to prove this theorem, we must develop the following.

First we need some definitions

\begin{defn}
Fix $d>0$.  A subset $S$ of $\mathbb{R}^d$ is said to be \textbf{determining} if whenever $U$ and $V$ are bounded $\mathbb{R}^d$-valued random varaibles such that $\forall s \in S,\, s \cdot U=_d \cdot V$ implies that $U=_d V$.
\end{defn}

As mentioned multiple times in the introduction, algebraic properties of certain restrictions will be important.  Thus it is no surprise that we need to consider the $0$ set of polynomials.

\begin{defn}
A subset $S$ of $\mathbb{R}^d$ is called a \textbf{projective variety} if there exists a degree-homogeneous polynomial $p$ defined on $\mathbb{R}^d$ such that $p\neq0$ and $p^{-1}(0)=S$.
\end{defn}

The next result, based on the work of Cuesta-Albertos, Fraiman, and Ransford in \cite{sharp} (Theorems 3.1 and 3.5), and the subsequent proposition will be used throughout this article.

\begin{lem}
Fix $d\ge1$.  A subset $S$ of $\mathbb{R}^d$ is determining if and only if it is not contained in any projective variety. \label{cramer}
\end{lem}

We will borrow the ideas found in \cite{sharp} in order to prove this lemma.

\begin{proof}
Fix $d\ge1$ and $S\subset \mathbb{R}^d$.  We assume that $S$ is nonempty because the claim is true if $S$ is empty.

Suppose $S$ is not contained in any projective variety, and consider two bounded random variables $U, V$.  For each $n\ge 0$, we define the polynomial 
\begin{equation}p(x)=\mathbb{E}[(x\cdot U)^n]-\mathbb{E}[(x\cdot V)^n]\end{equation} 
which is homogeneous of degree $n$. Since $p$  vanishes on $S$, it follows that $p$ must be the $0$ polynomial.  It follows that $\forall x \in \mathbb{R}^d$ we have that all the moments of the real-valued, bounded random variable $x\cdot U$ are equal to those of $x \cdot V$. Hence $\forall x\in \mathbb{R}^d$ we have $x\cdot U=_dx\cdot V$, thus $U=_dV$.

For the converse, it suffices to assume that $S$ is a projective variety and construct two different probability measures $\mu$ and $\nu$ with bounded support, defined on $\mathbb{R}^d$, such that $\forall s \in S,\, s \cdot \mu=s \cdot \nu$.  Here, if $U$ is $\mu$ distributed then $s\cdot \mu$ means the distribution probability measure of $s\cdot U$ and similarly for $\nu$.

Define an auxiliary function $f:\mathbb{C}^d\rightarrow \mathbb{C}$ given by
\begin{equation}f(z):=\prod_{j=1}^{d}\frac{\sin z_j-z_j}{z_j^3}.\end{equation}
Here, we have $z=(z_1, \dots, z_d)$.

It is routine to verify that 

(i)$f$ is even, entire, and real valued when restricted to $\mathbb{R}^d$

(ii)$f$ is of exponential type, i.e. there is a $C>0$ such that $\forall z$ we have $|f(z)|\le C e^{\sum_{j=1}^d |z_j|}$ 

(iii)There is a $C>0$ such that $\forall x \in \mathbb{R}^d$ we have $|f(x)|\le C/(1+||x||^2)$. 
 
(iv) we have $f(0)\neq0$.

Now, find $p$ a homogeneous polynomial on $\mathbb{R}^d$ that is not a constant such that $S\subset p^{-1}(0)$.  The reason $p$ can be chosen to be not a constant is that $S$ is nonempty. Define $g:\mathbb{R}^d \rightarrow \mathbb{R}$ by $g(x)=p(x)^2f(x)^K$, where $K$ is chosen large enough so that $g\in L^2(\mathbb{R}^d)$.  This is possible by (iii).  Let $h=\hat{g}$ be the Fourier transform of $g$.  By Plancherel's theorem we have that $h\in L^2(\mathbb{R}^d)$ and $h$ is real-valued since $g$ is even and real-valued.  Moreover, the Paley-Wiener theorem tells us that $h$ is supported in a compact subset of $\mathbb{R}^d$.  Here, we are applying Paley-Wiener to $g$ which is exponential type, analytic because $f$ is exponential type, analytic. (We extend the definition of $p$ to $\mathbb{C}^d$.)  But, depending on the convention of the Fourier transform, we already know the inverse Fourier transform of $g$ in terms of $h$.  This shows that $h$ is $0$ off of a compact set.

Define finite, positive Borel measures with compact support by 
\begin{equation}\mu=h^+dx,\, \nu=h^-dx,\end{equation}
which are mutually singular and therefore not equal.  Using the Fourier inversion theorem, there is a constant $c\neq0$ depending on the conventions of the Fourier transform and its inverse, such that

\begin{equation}\hat{\mu}-\hat{\nu}=cg=cp^2f^K \text{ on } \mathbb{R}^d.\end{equation}
Evaluating this equality at $0 \in \mathbb{R}^d$ we find that 

\begin{equation}\mu(\mathbb{R}^d)-\nu(\mathbb{R}^d)=cp(0)^2f(0)^K=0.\end{equation}  
Since $\mu$ and $\nu$ are both nonzero (because $g$ is not zero, $p$ is not zero, so $h$ is not zero) and have the same total mass, we may renormalize if necessary to force them to be probability measures.

We also have that for any $x \in \mathbb{R}^d$ $x\cdot \mu = x\cdot \nu$ if and only if $\forall t \in \mathbb{R}$ we have $\hat{\mu}(tx)-\hat{\nu}(tx)=0$ if and only if $\forall t \in \mathbb{R},\, cp(tx)^2f(tx)^K=0$.  Thus, $\forall x \in S$ we have $p(x)=0$ and thus $x\cdot \mu = x \cdot \nu$.

This completes the proof that $S$ is not determining, because $\mu$ and $\nu$ are distributions of bounded random vectors with values in $\mathbb{R}^d$ for which their projections along vectors in $S$ agree in law, but they are not equidistributed with one another.
\end{proof}

Part of the utility of presenting the above proof is to show how constructions in the sequel that use this lemma can be made explicit.  We have now stated and proven what we need from the existing literature.  We use the above to prove the following proposition, which will drive many constructions in this article.

\begin{prop}
There exist $U\neq_d V$ which are valued in the unit tetrahedron 

\begin{equation}T_4:=\{(x_0, x_1, x_2, x_3)|\forall j \in \{0, 1, 2, 3\},\, x_j\ge0,\, \sum_{j=0}^3x_j=1\}\subset \mathbb{R}^4\end{equation}
such that $\forall y \in \mathbb{R}$ we have $c_4(y)\cdot U=_d c_4(y) \cdot V$.  Here, $c_4(y)=(1, y, y^2, y^3)\in \mathbb{R}^4.$\label{example maker}
\end{prop}

\begin{proof}
Let $c_3(y)=(1, y, y^2) \in \mathbb{R}^3$ for all $y\in\mathbb{R}$.  By Lemma \ref{cramer}, there exist $W=(W_1, W_2, W_3)\neq_d Z=(Z_1, Z_2, Z_3)$ bounded $\mathbb{R}^3$-valued RVs such that 

\begin{equation}
\forall y \in \mathbb{R},\, c_3(y)\cdot W=_dc_3(y)\cdot Z. \label{construct}
\end{equation}
Let $H=\{(x, y, z)\in \mathbb{R}^3|1\ge x \ge y \ge z \ge 0\}$.  Observe that $H$ has nonempty interior.  Therefore, given any compact set $C\subset \mathbb{R}^3$, there exist $a \in \mathbb{R}$, $a \neq0$, $b\in \mathbb{R}^3$ such that $aC+b \subset H$.  Because \eqref{construct} is unchanged by rescaling and translation, we may assume that $W, Z$ are $H$-valued.  Define

\begin{align*}
U_0&=1-W_1 &V_0&=1-Z_1\\
U_1&=W_1-W_2 &V_1&=Z_1-Z_2\\
U_2&=W_2-W_3 &V_2&=Z_2-Z_3\\
U_3&=W_3-0 &V_3&=Z_3-0\\
U&=(U_0, U_1, U_2, U_3), &V&=(V_0, V_1, V_2, V_3).\numberthis
\end{align*}

Observe that $U$ and $V$ are $T_4$-valued.  From linear algebra (invertibility of the linear transform linking $(U_1, U_2, U_3)$ to $W$ and $(V_1, V_2, V_3)$ to $Z$) it follows that $(U_1, U_2, U_3)\neq_d(V_1, V_2, V_3)$ so that $U\neq_d V$.  By \eqref{construct} we have

\begin{equation}\forall y \in \mathbb{R},\, c_3(y)\cdot (U_1+U_2+U_3, U_2+U_3, U_3)=_dc_3(y)\cdot (V_1+V_2+V_3, V_2+V_3, V_3).\end{equation}
Plugging in the explicit form of $c_3(y)$ yields

\begin{equation}\forall y \in \mathbb{R},\, U_1+(1+y)U_2+(1+y+y^2)U_3=_dV_1+(1+y)V_2+(1+y+y^2)V_3.\end{equation}
Multiplying by $(y-1)$, we obtain

\begin{multline*}\forall y \in \mathbb{R},\, (y-1)U_1+(y^2-1)U_2+(y^3-1)U_3=_d\\ (y-1)V_1+(y^2-1)V_2+(y^3-1)V_3.\numberthis\end{multline*}
Thus, by definition of $U_0, V_0$ we learn that

\begin{equation}\forall y \in \mathbb{R},\, c_4(y)\cdot U=_dc_4(y)\cdot V.\end{equation}
We have shown that $U, V$ have the required properties.
\end{proof}

We are now ready to establish Theorem \ref{0123}.

\begin{proof}[Proof of Theorem \ref{0123}]
Obtain $U, V$ from Proposition \ref{example maker}.  Define $\alpha_1$ and $\alpha_2$ so that 

\begin{equation}U=_d\bigg(\alpha_1(\{0\}), \alpha_1(\{1\}), \alpha_1(\{2\}), \alpha_1(\{3\})\bigg)\end{equation}
and

\begin{equation}V=_d\bigg(\alpha_2(\{0\}), \alpha_2(\{1\}), \alpha_2(\{2\}), \alpha_2(\{3\})\bigg).\end{equation}
Strictly speaking, $\alpha_1, \alpha_2$ may not be directing measures for some exchangeable sequence, but the distribution of $\alpha_1, \alpha_2$ are still elements of $P(P(\mathbb{R}))$ and are thus mixing measures, from which we may extract directing measures that have the same distribution as $\alpha_1, \alpha_2$.  Thus, we may assume without loss of generality that $\alpha_1, \alpha_2$ are already defined on an appropriate probability space so that they are directing measures.  In the future, this argument will not be explicitly stated.

By a change of variables in $c_4(y)$ from Proposition \ref{example maker} we have $\forall s \ge0,$

\begin{align*}\mathcal{L}_{\alpha_1}(s)&=\\(e^{-0s}, e^{-1s}, e^{-2s}, e^{-3s})\cdot \bigg(\alpha_1(\{0\}), \alpha_1(\{1\}), \alpha_1(\{2\}), \alpha_1(\{3\})\bigg)&=_d\\ (e^{-0s}, e^{-1s}, e^{-2s}, e^{-3s})\cdot \bigg(\alpha_2(\{0\}), \alpha_2(\{1\}), \alpha_2(\{2\}), \alpha_2(\{3\})\bigg)&=\\\mathcal{L}_{\alpha_2}(s)\numberthis\end{align*}

However, we have that the corresponding $\Theta_1$ and $\Theta_2$ are distinct because $U$ and $V$ have different distributions, hence as do $\alpha_1$ and $\alpha_2$.  Thus, $\Theta_1, \Theta_2$ have the required properties when we use Lemma \ref{equiv} to prove that $F_{\{0, 1, 2, 3\}}$ is not good.
\end{proof}

\subsection{A Generalization of the Four-Value Case}

We may regard the values $0, 1, 2, 3$ as independent sums of $0$ copies of $1$, $1$ copy of $1$, $2$ copies of $1$ and $3$ copies of $1$ respectively.  We may replace the constant random variable $1$ with any nonnegative distribution to obtain a generalization of Theorem \ref{0123}.

\begin{prop}

Let $\mu \in P^+$ be nondegenerate and let $*$ denote convolution.  Let $L$ be the collection of convex combinations of $\delta_0, \mu, \mu*\mu, \mu*\mu*\mu$.  Then $L$ is not good.  The same is true of the convex combinations of $\delta_0, \mu, \mu*\mu, \mu*\mu*\mu, \dots$.
\label{g4}
\end{prop}

The primary purpose of this subsection is to prepare for comparisons and analogies with material from subsection \ref{poiss}, for which the case in which $\mu$ is Poisson is important.  The techniques themselves are not logical prerequisites for material in the sequel.

\begin{proof}

To imitate the last proof, we need a homeomorphism between $T_4$ and $L$.  We propose the map assigning a vector $(a, b, c, d)\in T_4$ the element $a\delta_0+b\mu+c\mu*\mu+d\mu*\mu*\mu$ of $L$.  This map is surjective. To see it is injective, if there are $(a, b, c, d), (a', b', c', d')$ with $a\delta_0+b\mu+c\mu*\mu+d\mu*\mu*\mu=a'\delta_0+b'\mu+c'\mu*\mu+d'\mu*\mu*\mu$ then taking Laplace transforms we have that $\forall s \ge 0,$ 

\begin{multline*}a\mathcal{L}_\mu(s)^0+b\mathcal{L}_\mu(s)^1+c\mathcal{L}_\mu(s)^2+d\mathcal{L}_\mu(s)^3=\\ a'\mathcal{L}_\mu(s)^0+b'\mathcal{L}_\mu(s)^1+c'\mathcal{L}_\mu(s)^2+d'\mathcal{L}_\mu(s)^3.\numberthis\label{iterate}\end{multline*}

Consider the operation of multiplying by $\mathcal{L}_\mu(s)$, then taking the derivative in $s$, and then dividing by $\frac{d}{ds}\mathcal{L}_\mu(s)$.  This division is legitimate because $\frac{d}{ds}\mathcal{L}_\mu<0$ for nondegenerate $\mu$.  Iterating this process on \eqref{iterate} as many times as we wish shows that $(a, b, c, d)=(a', b', c', d')$.

Bijective continuous maps between compact Hausdorff spaces are homeomorphisms.

Now, obtain $U, V$ as in Proposition \ref{example maker}.  Define 

\begin{equation}\alpha_1=U_0\delta_0+U_1\mu+U_2\mu*\mu+U_3\mu*\mu*\mu\end{equation}
and 
\begin{equation}\alpha_2=V_0\delta_0+V_1\mu+V_2\mu*\mu+V_3\mu*\mu*\mu.\end{equation}
We assume without loss of generality that $\alpha_1, \alpha_2$ are directing measures.
This way, $\Theta_1, \Theta_2$ are the pushforwards of $U, V$ via the above homeomorphism. Since $\Theta_1, \Theta_2$ are the distributions of $\alpha_1, \alpha_2$ by Remark \ref{corresp} we have that $\alpha_1\neq_d\alpha_2$ and $\Theta_1\neq\Theta_2$.  However, the random Laplace transform of $\alpha_1$ is

\begin{equation}\mathcal{L}_{\alpha_1}=\bigg(\mathcal{L}_\mu^0, \mathcal{L}_\mu^1, \mathcal{L}_\mu^2, \mathcal{L}_\mu^3\bigg)\cdot U\end{equation}
and the Laplace transform of $\alpha_2$ is
\begin{equation}\mathcal{L}_{\alpha_2}=\bigg(\mathcal{L}_\mu^0, \mathcal{L}_\mu^1, \mathcal{L}_\mu^2, \mathcal{L}_\mu^3\bigg)\cdot V.\end{equation}
Vectors of the form $\bigg(\mathcal{L}_\mu^0, \mathcal{L}_\mu^1, \mathcal{L}_\mu^2, \mathcal{L}_\mu^3\bigg)(s)$ are a subset of the image of $c_4$.  Thus, $\forall s$, $\alpha_1$ and $\alpha_2$ have random Laplace transforms evaluated at $s$ that have the same distribution.  Thus, by Lemma \ref{equiv} it follows that $L$ is not good.
\end{proof}

\subsection{Another Generalization of the Four-Value Case}

The primary purpose of this subsection is to prepare for comparisons and analogies with material from Subsection \ref{poiss}, and to expose some new techniques that are useful in proving results showing the relationship between arithmetic and algebraic dependences versus uniqueness results. The techniques themselves are not logical prerequisites for material in the sequel.

We may instead replace the role of independent sums by regular sums of the random variable with itself.  So we regard $0, 1, 2, 3$ as the sum of $0, 1, 2, 3$ copies of $1$.  If instead of $1$, we use an arbitrary nonnegative random variable, we arrive at the following generalization.

\begin{prop}
Let $\mu \in P^+$ be such that $\mathcal{L}_\mu(s)$ is a rational function in $s$, and let $T\neq0$ be a random variable with distribution $\mu$.  Fix $N\ge3$.  Let $L$ be the convex combinations of the distributions of $0T, 1T, 2T,\dots, NT$.  i.e. $L\subset P^+$ is the set of probability measures that have Laplace transform of the form

\begin{equation}\sum_{j=0}^{N}b_j\mathcal{L}_\mu(js).\end{equation}
\label{general}
Then $L$ is not good.
\end{prop}

Independent sums of exponential random variables give rational Laplace transforms, for example.

\begin{proof}
It suffices to handle the case $N=3$.  Assume that $\mathcal{L}_\mu(s)=p(s)/q(s)$ with $p, q$ polynomials sharing no common factor, and $q$ having no zeros in $[0, \infty)$.  We seek a homogeneous polynomial $r\neq0$ of $3$ variables for which 

\begin{equation}r\left(\frac{p(1s)}{q(1s)}-\frac{p(0s)}{q(0s)}, \frac{p(2s)}{q(2s)}-\frac{p(0s)}{q(0s)}, \frac{p(3s)}{q(3s)}-\frac{p(0s)}{q(30)}\right)=0.\end{equation}
This is equivalent to 

\begin{align*}r(p(1s)q(0s)q(2s)q(3s)-p(0s)q(1s)q(2s)q(3s), p(2s)q(0s)q(1s)q(3s)-\\p(0s)q(1s)q(2s)q(3s), p(3s)q(0s)q(1s)q(2s)-p(0s)q(1s)q(2s)q(3s))=0\numberthis\end{align*}
Let us say that $p$ has degree $n\ge0$, $q$ has degree $m\ge0$, and $r$ has degree $l>0$.  The space of polynomials in $s$ of degree at most $l(3m+n)$ has dimension linear in $l$ as a vector space over $\mathbb{R}$ whereas the space of polynomials in $3$ variables that are homogeneous, of degree $l$ is quadratic in $l$.  Therefore, there exists $l$ large enough such that the assignment of homogeneous degree $l$ polynomials in $3$ variables to polynomials of degree at most $l(3m+n)$ given by

\begin{align*}r \mapsto r(p(1s)q(0s)q(2s)q(3s)-p(0s)q(1s)q(2s)q(3s), p(2s)q(0s)q(1s)q(3s)-\\p(0s)q(1s)q(2s)q(3s), p(3s)q(0s)q(1s)q(2s)-p(0s)q(1s)q(2s)q(3s))\numberthis\end{align*}

has nontrivial kernel.

Thus we have a nonzero homogeneous polynomial $r$ for which $r(\mathcal{L}_\mu(1s)-\mathcal{L}_\mu(0s), \mathcal{L}_\mu(2s)-\mathcal{L}_\mu(0s), \mathcal{L}_\mu(3s)-\mathcal{L}_\mu(0s))=0$.  

Now, we use Lemma \ref{cramer} to find $W=(W_1, W_2, W_3)\neq_dZ=(Z_1, Z_2, Z_3)$ bounded random vectors for which $\forall s\ge0$ we have

\begin{align*}(\mathcal{L}_\mu(1s)-\mathcal{L}_\mu(0s), \mathcal{L}_\mu(2s)-\mathcal{L}_\mu(0s), \mathcal{L}_\mu(3s)-\mathcal{L}_\mu(0s))\cdot W=\\(\mathcal{L}_\mu(1s)-\mathcal{L}_\mu(0s), \mathcal{L}_\mu(2s)-\mathcal{L}_\mu(0s), \mathcal{L}_\mu(3s)-\mathcal{L}_\mu(0s))\cdot Z.\numberthis\end{align*}

We then find $C$ compact such that $W, V$ are both $C$-valued, and $a\neq0$, $b \in \mathbb{R}^3$ such that $aC+b\subset T'_3:=\{(a, b,c)|a+b+c\le1, a, b, c\ge0\}$.  This is possible because $C$ is compact and $T'_3$ has nonempty interior.  Thus, we may assume that $W, Z$ were $T'_3$-valued to begin with.  We now define $U, V$ via

\begin{align*}
U_0&=1-U_1-U_2-U_3 &V_0&=1-Z_1-Z_2-Z_3\\
U_1&=W_1 &V_1&=Z_1\\
U_2&=W_2 &V_2&=Z_2\\
U_3&=W_3 &V_3&=Z_3\\
U&=(U_0, U_1, U_2, U_3), &V&=(V_0, V_1, V_2, V_3).\numberthis
\end{align*}

Thus $U\neq_d V$ and $\forall s\ge0$, we have

\begin{multline*}(\mathcal{L}_\mu(0s), \mathcal{L}_\mu(1s), \mathcal{L}_\mu(2s), \mathcal{L}_\mu(3s))\cdot U=_d\\(\mathcal{L}_\mu(0s), \mathcal{L}_\mu(1s), \mathcal{L}_\mu(2s), \mathcal{L}_\mu(3s))\cdot V.\label{lap}\numberthis\end{multline*}
Let $\mu_k$ denote the probability distribution of $kT$.  We define $\alpha_1=U_0\mu_0+U_1\mu_1+U_2\mu_2+U_3\mu_3$ and $\alpha_2=V_0\mu_0+V_1\mu_1+V_2\mu_2+V_3\mu_3$ which are  $L$-valued.  Without loss of generality, we assume $\alpha_1$ and $\alpha_2$ are directing measures.

Since $T\neq0$, we have that all of the $\mu_k$ are distinct, nondegenerate, and therefore have Laplace transforms with derivatives that are never $0$.  We aim to show that $T_4$ is homeomorphic to $L$ through the map $(a, b, c, d)\mapsto a\mu_0+b\mu_1+c\mu_2+d\mu_3$.   This is surjective.  Also, it is injective because if there are $(a, b, c, d), (a', b', c', d',)$ such that $a\mu_0+b\mu_1+c\mu_2+d\mu_3=a'\mu_0+b'\mu_1+c'\mu_2+d'\mu_3$ then we may take Laplace transforms to obtain $\forall s \ge0$

\begin{multline*}a\mathcal{L}_\mu(0s)+b\mathcal{L}_\mu(1s)+c\mathcal{L}_\mu(2s)+d\mathcal{L}_\mu(3s)=\\a'\mathcal{L}_\mu(0s)+b'\mathcal{L}_\mu(1s)+c'\mathcal{L}_\mu(2s)+d'\mathcal{L}_\mu(3s).\numberthis\end{multline*}
We may take the derivative of this relation $k$ times, then take $s\downarrow 0$, then divide by $[\frac{d^k}{ds^k}\mathcal{L}_\mu](0)$.  Again, this operation shows us that $(a, b, c, d)=(a', b', c', d')$.  Continuous bijections between compact Hausdorff spaces are always homeomorphisms.

Therefore, $U\neq_dV$ implies that $\alpha_1\neq_d\alpha_2$ and $\Theta_1\neq\Theta_2$. Also, $\forall s \ge0,\, \mathcal{L}_{\alpha_1}(s)=_d\mathcal{L}_{\alpha_1}(s)$.  This is because the left side is the left side of \eqref{lap} and the right side is the right side of \eqref{lap}.  This suffices by Lemma \ref{equiv}.

\end{proof}

The answer to the Aldous problem actually changes despite the fact that the arithmetic dependences in some sense still remain.  To show this, roughly speaking we will use a very transcendental Laplace transform to make the arithmetic dependences irrelevant.

\begin{prop}
Let $T$ be Poisson distributed with parameter $\lambda$.  Let $N\ge1$.  Let $\mu$ be the distribution of $T$, and define $\mu_0, \mu_1, \mu_2, \dots, \mu_N$ as before.  Define $L$ as in Proposition \ref{general}.  Then $L$ is good. \label{Poisson}
\end{prop}

\begin{proof}
We will argue for $N=3$, with the general case being similar.  Let $\mathcal{L}(s)=e^{\lambda(e^{-s}-1)}$ denote the Laplace transform of $\mu$.  Because we already know that $T_4$ is homeomorphic to $L$ in the natural way (see the last proof), it suffices to show that there cannot be any $U=(U_0, U_1, U_2, U_3)\neq_d V=(V_0, V_1, V_2, V_3)$ defined on $T_4$ for which $\forall s \ge0,\, (\mathcal{L}(0s), \mathcal{L}(1s), \mathcal{L}(2s), \mathcal{L}(3s))\cdot U=(\mathcal{L}(0s), \mathcal{L}(1s), \mathcal{L}(2s), \mathcal{L}(3s))\cdot V$.  We will in fact show that there is no homogeneous polynomial other than $0$ that vanishes on the image of the curve $(\mathcal{L}(0s), \mathcal{L}(1s), \mathcal{L}(2s), \mathcal{L}(3s))\in \mathbb{R}^4$ defined for $s\ge0$. Suppose that $r\neq0$ is such a homogeneous polynomial of degree $l$, say.  Order the set $M_l$ of monic monomials of total degree $l$ in $4$ variables by ordering lexicographically on the exponents, with the fourth variable taking highest priority, then the third, second, then first.  This is a total ordering.  Write $a_w$ for the coefficient of any monic monomial $w$ in $r$.  Find the largest monic monomial with a nonzero coefficient in $r$.  Call this monomial $m(x_0, x_1, x_2, x_3)$.  Then we have

\begin{equation}r(x_0, x_1, x_2, x_3)=a_mm(x_0, x_1, x_2, x_3)+\sum_{w<m\in M_l}a_ww(x_0, x_1, x_2, x_3)\end{equation}
with $a_m\neq 0$.

We have $\forall s \ge0$
\begin{equation}r(\mathcal{L}(0s), \mathcal{L}(1s), \mathcal{L}(2s), \mathcal{L}(3s))=0\label{extend}\end{equation}
so by the theorem of complex analysis asserting the equality of holomorphic functions defined on the same connected open domain, agreeing on a set with an accumulation point within this domain, \eqref{extend} holds also for $s<0$.  The term of $r(\mathcal{L}(0s), \mathcal{L}(1s), \mathcal{L}(2s), \mathcal{L}(3s))$ corresponding to $m$ goes to $\infty$ as $s\rightarrow -\infty$ faster than any of the other terms, so the coefficient $a_m$ is $0$, contradiction.
\end{proof}

\subsection{The Normal Case}

The normal case is another case of significance to the next section on the continuous time analog of the present problem.

\begin{lem} 

Let $N$ be the collection of normal distributions, including the degenerate ones.  For all $\mu \in N$ let $M(\mu)=$mean of $\mu$ and let $V(\mu)=$ variance of $\mu$.  Then the map $(M, V): N\rightarrow \mathbb{R} \times [0, \infty)$ is a homeomorphism, hence measurable isomorphism.  \label{homeo}

\end{lem}

\begin{proof}
This follows from convergence of types.
\end{proof}

\begin{rem}
Let $\Theta$ be a mixing measure supported in $N$.  Since $\Theta$ is then a probability measure on $N$, we can view $M, V$ as random variables giving the (random) mean and variance of an element of $N$ drawn with prior distribution $\Theta$.  Then the joint distribution of the corresponding exchangeable sequence $X$ is given by $(X_i)_{\{i \in \mathbb{N}\}}=_d(A+B^{1/2}\mathcal{N}(0, 1)_i)_{\{i\in\mathbb{N}\}}$ where the entire family $\{\mathcal{N}(0, 1)_1, \mathcal{N}(0, 1)_2, \dots, (A, B)\}$ is independent (but the notation indicates $A, B$ may not be independent), $(A,B)=_d (M,V)$, and $\mathcal{N}(0, 1)_i$ is normal with mean $0$ and variance $1$.
\end{rem}

See, for instance, p.29 of \cite{A} regarding this remark.

We highlight in the remark that the distribution of $(M, V)$ is calculated relative to $\Theta$, and that it is necessary to use $(A, B)$ instead of $(M, V)$ when we deal with the independent normals because these normals and $(A, B)$ are constructed on the same probability space.  

We already specified that $X, Y$ corresponds to $\Theta_1, \Theta_2$ and $\alpha_1, \alpha_2$ via Remark \ref{corresp}.  For this subsection, when $\Theta_1, \Theta_2$ are supported on $N$, we will use  $(M_1, V_1)$ to indicate $(M, V)$ defined on the probability space $(N, \Theta_1)$ and $(M_2, V_2)$ to indicate $(M, V)$ defined on the probability space $(N, \Theta_2)$.

We have the following transform inversion fact.

\begin{lem}
Suppose $\mu, \nu$ be probability measures on $\mathbb{R}\times [0, \infty)$ such that $\forall t, s \in \mathbb{R}\times[0, \infty)$ we have 

\begin{equation}\int_{\mathbb{R}\times[0,\infty)}e^{itx-sy}d\mu(x,y)=\int_{\mathbb{R}\times[0,\infty)}e^{itx-sy}d\nu(x,y).\end{equation}
Then $\mu=\nu$.\label{hybrid}
\end{lem}

Notice that $N$ is not a subset of $P^+$.  The author is uncertain if $N$ is good or not, which seems to rely on a generalization of Muntz's Theorem (see \cite{Muntz}) which would include the sequence of points $t_n=1/n$ in the role of the values of the parameter at which the Laplace transform is known a priori.  However, what is true is the following.

\begin{prop}
Let $\Theta_1, \Theta_2$ supported in $N$ be given.  Suppose that $V_1, V_2$ have finite MGF in some neighborhood around $0$, and that $\forall n >0,\, S_n=_dT_n$.  Then $\Theta_1=\Theta_2$. \label{normal}
\end{prop}

\begin{proof}

Let $(A, B), (A', B')$ have the same distributions as $(M_1, V_1), (M_2, V_2)$ respectively with the three random vectors/variables $(A, B), (A', B'), \mathcal{N}(0, 1)$ all independent.  From $S_n=_dT_n$ we learn that $\forall n \ge 0,\, nA+(nB)^{1/2}\mathcal{N}(0, 1)=_dnA'+(nB')^{1/2}\mathcal{N}(0, 1)$  Computing the characteristic function of both sides reveals that $\forall n \ge0, t \in \mathbb{R}$

\begin{equation}\mathbb{E}[e^{itnA-t^2nB/2}]=\mathbb{E}[e^{itnA'-t^2nB'/2}]\end{equation}
or equivalently $\forall n >0,\, t\in \mathbb{R}$

\begin{equation}\mathbb{E}[e^{itA-\frac{t^2B}{2n}}]=\mathbb{E}[e^{itA'-\frac{t^2B'}{2n}}].\label{characteristic}\end{equation}
Looking at \eqref{characteristic} for fixed $t$ and varying $n$, it follows that the convergence of the MGF in a neighborhood of $0$ is precisely what is needed to be able to use complex analysis to conclude that for each fixed $t$, we have that $\forall s \ge0$

\begin{equation}\mathbb{E}[e^{itA-sB}]=\mathbb{E}[e^{itA'-sB'}]\end{equation}
Particularly, we are using the fact that holomorphic functions defined on a common connected open domain, agreeing on a set with a limit point in the domain must be equal.  Namely, this limit point would be $s=0$ regardless of which $t$ was fixed.  The MGF hypothesis is what allows $s=0$ to be in the (interior of the) domain of these transforms.

Then, by Lemma \ref{hybrid} it follows that $(A, B)=_d(A', B')$ so that $(M_1, V_1)=_d(M_2, V_2)$ from which it follows by Lemma \ref{homeo} that $\Theta_1=\Theta_2$.

\end{proof}

\section{Continuous Time Exchangeability Problem}

We could view the questions answered in the last section from the perspective of $S_n, T_n$.  These are mixtures of partial sums of iid sequences.  From this point of view, it is natural to consider mixtures of L\'evy Processes, which are the continuous time analog.  Recall that a L\'evy process is an independent stationary increments process that is continuous in probability and starts at $0$.  We will use the notation $S_t, T_t$ for mixtures of L\'evy processes, after they are defined, in order to reflect this analogy.  In order to aid our discussion, we recall:

\begin{lem}[L\'evy Khintchine Formula]
Let $Z=(Z_t)_{t\ge0}$ be a L\'evy process.  Then there exist unique $\beta, \sigma^2, \nu$ such that $\nu$ is a finite measure on $\mathbb{R}$ with $\nu(\{0\})=0$, $\sigma^2\ge0$, $\beta \in \mathbb{R}$ and $\forall u\in\mathbb{R}, t\ge0$ we have

\begin{equation}\mathbb{E}[e^{iuZ_t}]=\exp\left\{iut\beta-\frac{u^2t\sigma^2}{2}+t\int_{\mathbb{R}}\left(e^{iux}-1-\frac{iux}{1+x^2}\right)\frac{1+x^2}{x^2}d\nu(x)\right\}\label{Levy}.\end{equation}
Furthermore, every $\beta \in \mathbb{R}, \sigma^2\ge0, \nu$ a finite measure on $\mathbb{R}$ with no atom at $0$ corresponds to a unique (up to distributional equality) L\'evy process with characteristic function given by \eqref{Levy}.
\end{lem}

We would now like to define the notion of a mixture of L\'evy processes.  For technical reasons, we downplay the role of exchangeability.  For the moment, we also focus on characteristic functions in order to be able to state the definition before worrying about measurability concerns associated with generalizing \eqref{iid-theta}.  Then we will show how the definitions we make are directly analogous to those of the previous section.  These claims are mostly for checking intuition about what a mixture of L\'evy Processes should mean, but they will also be used in Subsection \ref{poiss}. (They will not be featured as prominently in the next subsection.)

Given parameters $\beta\in\mathbb{R}, \sigma^2\ge0, \nu$ finite measure on $\mathbb{R}$ with no atom at $0$, we will use the notation $\phi_{\beta, \sigma^2, \nu, t_1, \dots, t_n}$ for the $n$-variate characteristic function of $(Z_{t_1}, \dots, Z_{t_n})$ where the L\'evy process $Z_t$ is chosen for parameters $\beta, \sigma^2, \nu$.  We will use $M_0^+$ to denote the collection of nonnegative finite measures on $\mathbb{R}$ that vanish at $\{0\}$. From now on, we will always implicitly assume $(\beta, \sigma^2, \nu)\in \mathbb{R}\times [0,\infty)\times M_0^+$.  We will use $L_{\beta, \sigma^2, \nu}:=\bigg((Z_t)_{t\ge0}\bigg)^*(\mathbb{P})$ to mean the pushforward of $\mathbb{P}$ (the probability measure on whichever space the process under study is defined on) via the L\'evy process $Z=(Z_t)_{t\ge0}$, i.e. the (joint) distribution of the L\'evy Process.

\begin{defn}

A \textbf{mixture of L\'evy processes} is $S=(S_t)_{t\ge0}$ such that there exist a probability measure $\Theta$ on $\mathbb{R}\times [0,\infty)\times M_0^+$ for which the joint characteristic function of $S$ is specified by $\forall n\ge1, \forall 0\le t_1\le \dots \le t_n, u_1, \dots, u_n \in \mathbb{R}$ we have

\begin{equation}\mathbb{E}[e^{i\sum_{j=1}^n u_jS_{t_j}}]=\int_{\mathbb{R}\times [0,\infty)\times M^+} \phi_{\beta, \sigma^2, \nu, t_1, \dots, t_n}(u_1, \dots, u_n) d\Theta (\beta, \sigma^2, \nu)\label{chf}\end{equation}
We call $\Theta$ the mixing measure.

\end{defn}

By using discrete time De Finetti, it follows that in this case $\Theta$ is uniquely determined by the distribution of $S$.

In the discrete time case, we were able to obtain a discrete time process (namely $S_n$) from the mixing measure $\Theta$.  It is reasonable to ask if the same can be done here.

\begin{lem}
Given a probability measure $\Theta$ on $\mathbb{R}\times [0,\infty)\times M_0^+$, there is a unique (up to joint distributional equality) stochastic process $S=(S_t)_{t\ge0}$ for which $\Theta$ is the mixing measure.
\end{lem}

\begin{proof}
First we check existence.  Restrict to finite dimensional distributions, using \eqref{chf} to define these finite dimensional distributions.  Then check Kolmogorov consistency and use Kolmogorov extension theorem.

The uniqueness up to distributional equality is built into the definition of mixture of L\'evy processes.
\end{proof}

From L\'evy continuity, stationarity of increments and the fact that distributional convergence to $0$ is the same as in probability convergence to $0$, it follows that all mixtures of L\'evy processes are continuous in probability.

We will use the following notation: $S=(S_t)_{t\ge0}, T=(T_t)_{t\ge0}$ will be the mixture of L\'evy processes, with mixing measures $\Theta_1, \Theta_2$.  We will have no need for trying to define some analog of $\alpha_1, \alpha_2$ in this context.

The set $\mathcal{I}$ of infinitely divisible distributions is closed in $P(\mathbb{R})$, hence measurable.

We regard \eqref{Levy} as specifying a bijection between $\mathbb{R}\times [0,\infty)\times M_0^+$ and the collection $\mathfrak{L}$ of distributions of L\'evy processes $\bigg((X_t)_{t\ge0}\bigg)^*(\mathbb{P})$.  So $\mathfrak{L}\subset P(\mathbb{R}^{[0,\infty)})$.  
$\mathfrak{L}$ is given the smallest $\sigma$ algebra so that passage from an element of $\mathfrak{L}$ to its marginals is measurable from $\mathfrak{L}$ to $\mathcal{I}$.  That is, $\forall t_0\ge 0,\, (X_t)_{t\ge0}^*(\mathbb{P})\mapsto X_{t_0}^*(\mathbb{P})$ should be measurable. It follows from standard proofs of \eqref{Levy} that the bijection specified by \eqref{Levy} is a measurable isomorphism.  There is also a natural measurable isomorphism between $\mathfrak{L}$ and $\mathcal{I}$ via $\bigg((X_t)_{t\ge0}\bigg)^*(\mathbb{P})\mapsto X_1^*(\mathbb{P})$.  

Because $P(\mathbb{R})$ with vague convergence is a Polish space, and $\mathcal{I}$ is closed in $P(\mathbb{R})$, we have that $\mathfrak{L}$ is a standard Borel space.  Therefore, our definition of a mixture of L\'evy processes is entirely parallel to the notion of mixture from the last section.

Because of these observations, it is sensible to state and we have proven the following:

\begin{lem}
$S$ is a mixture of L\'evy processes if and only if there exists $\Theta$ a probability measure on $\mathfrak{L}$ for which $\forall A\subset \mathbb{R}^{[0,\infty)}$ product measurable,

\begin{equation} 
\mathbb{P}(S\in A)=\int_\mathfrak{L}\gamma(A)d\Theta(\gamma)
\end{equation}
if and only if there exists $\Theta$ a probability measure on $\mathbb{R}\times [0,\infty)\times M_0^+$ for which $\forall A\subset \mathbb{R}^{[0,\infty)}$ product measurable,
\begin{equation}
\mathbb{P}(S\in A)=\int_{\mathbb{R}\times [0,\infty)\times M_0^+} L_{\beta, \sigma^2, \nu}(A)d\Theta(\beta, \sigma^2, \nu).
\end{equation}

In any case, $\Theta$ is unique.
\end{lem}

Thus, when speaking of $\Theta$ being a mixing measure or related topics, we will freely use these identifications.  For example, we will allow ourselves to say ``mixtures of Brownian motions''.  Also, we will no longer use the notation $\mathbb{R}\times [0,\infty)\times M_0^+$ and will use $\mathfrak{L}$ instead.  These identifications needed to be measurable in order for it to be possible to discuss mixtures using any of the descriptions, reconciling with the intuition that L\'evy processes are truly the same as their L\'evy Khintchine parameters and as infinitely divisible distributions.

The interested reader can combine what we have done so far with \cite{cts} to see that being a mixture of L\'evy processes is equivalent to being continuous in probability and satisfying a certain kind of exchangeable increments hypothesis.


Again, we will use a notion of goodness to abbreviate our discussion.

\begin{defn}
We will say that a measurable subset $L$ of $\mathfrak{L}$ is \textbf{good} if whenever $\Theta_1, \Theta_2$ are concentrated on $L$ and $\forall t\ge0,\, S_t=_dT_t$ we have $S=_dT$.
\end{defn}

\subsection{The Case of Brownian Motions}

Recall that a Brownian motion is a Gaussian L\'evy process, and can have drift and can proceed at any positive rate. (i.e. we only require that the variance at $t=1$ is positive.)  We denote the space of Brownian motions by $\text{BM}\subset \mathfrak{L}$.  BM corresponds to the requirement that the $\nu$ component of the L\'evy Khintchine formula is $0$.  We claim that

\begin{prop}

$\text{BM}$ is good.

\end{prop}

\begin{proof}
Using \eqref{chf} for one value of $t$ at a time, we have $\forall u\in \mathbb{R}, t\ge0$

\begin{align*}\int_{\mathbb{R}\times[0,\infty)\times\{0\}}\exp\{iut\beta-tu^2\sigma^2/2\}d\Theta_1(\beta, \sigma, 0)=\\\int_{\mathbb{R}\times[0,\infty)\times\{0\}}\exp\{iut\beta-tu^2\sigma^2/2\}d\Theta_2(\beta, \sigma, 0).\numberthis\end{align*}
Lemma \ref{hybrid} now finishes the proof.

\end{proof}

Notice that in the discrete time normal case, we had the last equation only for $t=1/n$ but now we have it for all $t\ge0$, which is important in eliminating the need for assumptions about convergence of MGFs.  Because the discrete set of rationally related numbers $1/n$, arising via application of $r\mapsto 1/r$ to $\mathbb{N}$ in the proof of Proposition \ref{normal}, is replaced with a continuum in the above proof, this can be thought of as a destruction of the arithmetic structure.  As promised, this is a case in which passage to the continuous time problem implies not only the additional information of infinite divisibility, but crucially the observations at a continuum of times rather than only a discrete set.

\subsection{A Poisson-Flavored Case}

All functions of $u$ of the form

\begin{equation}\exp\left\{\int_{\mathbb{R}}(e^{iux}-1)d\mu(x)\right\}\label{special}\end{equation}
are characteristic functions of infinitely divisible distributions, as long as $\mu$ is a finite nonnegative Borel measure.  This can be seen by taking a vague limit of sums of independent Poisson Processes with various rates and jump sizes.


Call $\text{LISPP}$ the subset of $\mathfrak{L}$ determined by \eqref{special}. (Here $\text{LISPP}$ stands for ``limits of independent sums of Poisson Processes.'')  We may think of the elements of $\text{LISPP}$ as ``independent integrals'' of Poisson Processes, which is a different notion than a mixture of Poisson Processes and also different from compound Poisson processes.  By a calculation, we have

\begin{lem}
$\text{LISPP}$ is measurable in $\mathfrak{L}$ because it is actually determined by the conditions $\int_{\mathbb{R}}\frac{1+x^2}{x^2}d\nu(x)<\infty, \beta=\int_{\mathbb{R}} \frac{1}{x} d\nu(x)$.
\end{lem}

It follows from the description of $\text{LISPP}$ above that $\mu$ is uniquely determined by the infinitely divisible distribution.  Moreover, 

\begin{equation}(X_t)_{t\ge0}^*(\mathbb{P})\in \text{LISPP} \mapsto \mu \in M_+(\mathbb{R})\end{equation}
is a measurable isomorphism, which is defined on $\text{LISPP}$.  Therefore, we may identify each element of $\text{LISPP}$ with a nonnegative finite measure on $\mathbb{R}$ via this correspondence.  Also, if $\text{LISPP}^+$ is the subset of $\text{LISPP}$ corresponding to $\mu$ supported in $[0,\infty)$ (i.e. we only allow positive jump size Poisson Processes to enter the independent integral), then $\text{LISPP}^+$ is of course measurable in $\text{LISPP}$.

Since \eqref{special} specifies the distribution of a nonnegative infinitely divisible distribution for elements of $\text{LISPP}^+$, the Laplace transform can be calculated by analytic continuation: $\forall \mu$ finite Borel measure on $[0,\infty)$, we have the function

\begin{equation}\exp\left\{\int_{[0,\infty)}(e^{-sx}-1)d\mu(x)\right\}\end{equation}
of $s$ is the Laplace transform of a member of $\text{LISPP}^+$, and moreover these are the only Laplace transforms of members of $\text{LISPP}^+$.

We will also refer to $\text{LISPP}_1, \text{LISPP}_1^+$ to denote the requirement that $\mu$ be a probability measure.  These are also measurable subsets of $\text{LISPP}$.

If $\Theta$ is concentrated on $\text{LISPP}, \text{LISPP}^+, \text{LISPP}_1$ or $\text{LISPP}_1^+$, then $\mu$ can be regarded as a random measure defined on $\text{LISPP}, \text{LISPP}^+, \text{LISPP}_1$ or $\text{LISPP}_1^+$.  We now show how the last section on the discrete problem can be embedded into the current problem.

\begin{lem}
Let $L$ be a measurable subset of $\text{LISPP}_1^+$.  Then $L$ is good if and only if $\forall \Theta_1, \Theta_2$ concentrated on $L$, we have

\begin{equation}\forall s\ge0,\, \mathcal{L}_{\mu_1}(s)=_d\mathcal{L}_{\mu_2}(s)\end{equation}
implies

\begin{equation}\Theta_1=_d\Theta_2.\end{equation}
Here, we regard $\mu \mapsto \mu$ as a mapp from $\text{LISPP}_1^+$ to $P([0,\infty))$, and we regard $\Theta_1, \Theta_2$ as giving the structure of a probability space to $\text{LISPP}_1^+$ in two different ways, so we require that $\mu_1, \mu_2$ are random probability measures with the same distribution as $\mu$ under $\Theta_1$ and $\Theta_2$ respectively.\label{newequiv}
\end{lem}

Therefore, even though there is no ideological connection between the mixing measure of the last section and the measures $\mu$ from this section which are similar to the L\'evy Khintchine measure, at the level of the mathematical formalisms the problems are related.

\begin{proof}

Observe first that if $S, T$ are mixtures from $L$, then $\forall t \ge0,\, S_t=_dT_t$ if and only if $\forall t\ge0, s\ge0, n\ge0$ we have

\begin{equation}\mathbb{E}[\exp\{nt\int_{[0,\infty)}(e^{-sx}-1)d\mu_1(x)\}]=\mathbb{E}[\exp\{nt\int_{[0,\infty)}(e^{-sx}-1)d\mu_2(x)\}]\label{PL}\end{equation}
by using \eqref{chf} for one value of $t$ at a time.  Then we know that \eqref{PL} is equivalent to $\forall t\ge0, s\ge0, n\ge0$

\begin{equation}\mathbb{E}[\exp\{t\int_{[0,\infty)}(e^{-sx}-1)d\mu_1(x)\}^n]=\mathbb{E}[\exp\{t\int_{[0,\infty)}(e^{-sx}-1)d\mu_2(x)\}^n]\end{equation}
which, by the bounded moment problem is equivalent to $\forall t\ge0, s\ge0$

\begin{equation}\exp\{t\int_{[0,\infty)}(e^{-sx}-1)d\mu_1(x)\}=_d\exp\{t\int_{[0,\infty)}(e^{-sx}-1)d\mu_2(x)\}.\end{equation}
But this last statement is equivalent to $\forall t\ge0, s\ge0$

\begin{equation}t\int_{[0,\infty)}(e^{-sx}-1)d\mu_1(x)=_dt\int_{[0,\infty)}(e^{-sx}-1)d\mu_2(x)\end{equation}
which is the same as $\forall s\ge0$

\begin{equation}\int_{[0,\infty)}(e^{-sx}-1)d\mu_1(x)=_d\int_{[0,\infty)}(e^{-sx}-1)d\mu_2(x)\end{equation}
and therefore also the same as $\forall s\ge0$

\begin{equation}\int_{[0,\infty)}e^{-sx}d\mu_1(x)=_d\int_{[0,\infty)}e^{-sx}d\mu_2(x)\end{equation}
because $\mu_1, \mu_2$ are always probability measures.


We have that $S=_dT$ if and only if $\Theta_1=\Theta_2$. 

\end{proof}

\begin{rem}

Notice how it did not matter that we made a continuum of observations because the nonnegativity assumption $\text{LISPP}^+$ allowed us to use the Laplace transform.  Since, when restricted to real arguments, exponentiation is invertible, we were able to cancel an exponentiation and then cancel the $t$.  Therefore, a measurable subset $\text{LISPP}_1^+$, when viewed as a subset of $\mathcal{I}$, is good if and only if it is good as in the last section.  That is, to tell apart two mixtures of $\text{LISPP}_1^+$s, we only need to observe at natural number times.  This manipulation was not available in the $\text{BM}$ case because there we were dealing with complex exponentiation, which also forbids the use of the bounded moment problem above.

\end{rem}

For the next result, $\forall A \subset P(\mathbb{R})$ measurable, we use 
\begin{equation}\text{LISPP}(A):=\{(X_t)_{t\ge0}^*(\mathbb{P}) \in \text{LISPP}|\mu\in A\}\end{equation}
where $\mu$ is the measure in \eqref{special} giving the characteristic function of $X_1$.  For example, $\text{LISPP}F_{\{0, 1, 2, 3\}}$ denotes the collection of L\'evy processes that can be written as an independent sum of of a Poisson Process of rate $0$ and jump size $x_0$, one of rate $1$ and jump size $x_1$, one of rate $2$ and jump size $x_2$, and one of rate $3$ and jump size $x_3$ such that $(x_0, x_1, x_2, x_3)\in T_4$.  Since we will always use $A$ such that all measures in $A$ are supported in $[0,\infty)$, our $\text{LISPP}(A)$ will always be a subset of $\text{LISPP}_1^+$ so we will be able to use the above lemma.

The fact that the continuum of observations does not destroy any arithmetic structure suggests that the situation with $\text{LISPP}_1^+$ will be more nuanced than the situation with $\text{BM}$, where uniqueness of $S=(S_t)_{t\ge0}$ held without further conditions.  Indeed, we now know that the variety of possibilities of goodness is at least as much as that of the discrete time problem:

\begin{thm}
Let $A\subset P^+$ be measurable, and consist only of compactly supported measures.  Then $A$ is good (in the only sense that is available, i.e. from the discrete time problem) iff $\text{LISPP}(A)$ is good (in either equivalently the sense of the present section or the last when $\text{LISPP}(A)$ is viewed as a subset of $\mathcal{I}\subset P(\mathbb{R})$).\label{bridge}
\end{thm}

\begin{proof}
This is a consequence of Lemma \ref{newequiv}
\end{proof}

\begin{rem}

The upshot of this theorem is that any set of probability measures that is good when in the role of the allowed components of the mixture are also good when in the role of $\mu$.  One could iterate this.  If $A$ is good, then $\text{LISPP}(A)$ is good, then $\text{LISPP}(\text{LISPP}(A))$ is good and so on.  After all, thanks to the fact that knowing each $S_n$ is enough, as long as we only concern ourselves with mixtures of L\'evy processes, there is no difference between the continuous time and discrete time problems.
\end{rem}

\begin{cor}
Let $\nu_1, \nu_2,\nu_3$ be probability measures on $[0,\infty)$. Let $C$ be the convex hull of $\{\nu_1, \nu_2, \nu_3\}$.  Then $\text{LISPP}(C)$ is good.  Also, if $A$ is a discrete, countable set of real numbers that is linearly independent over $\mathbb{Q}$ then $\text{LISPP}(F_A)$ is good.  $\text{LISPP}(F_{\{0,1,2,3\}})$ is not good. 
\end{cor}

\begin{rem}

Of course, most of the other results of the last section could be generalized just as easily, but they are not as meaningful as the ones listed above in the context of mixtures of L\'evy processes.

\end{rem}

\label{poiss}

Notice the comparison with Remark \ref{BGEP1}, which is a reasonable comparison because we already saw that in this case knowing all $S_t$ is no different than knowing only the $S_n$.  In Remark \ref{BGEP1} we were concerned with mixtures of Poissons (with jump size $1$ and rate $\lambda$) and we were mixing over different rates.  There, uniqueness of $S_n$ was true, which can also be deduced from our present machinery by applying Theorem \ref{bridge} to the  good set $A=\{\delta_x|x\in[0,\infty)\}.$ 

In general, uniqueness fails in the present setting for even jump sizes restricted to $\{0, 1, 2, 3\}$, because we allow the rates to vary as long as they add up to $1$.  Both this and Remark \ref{BGEP1} are different than the situation in Proposition \ref{Poisson}, where the rate was fixed, but the jump size was allowed to be $0, 1, 2,$ or $3$ and we were allowed to take convex combinations of these distributions before mixing them.

Another different situation occured in the context of Proposition \ref{g4} applied to Poisson distributions, where the jump size was fixed at $1$, the rate was fixed at $\lambda$, but we allowed independent sums and convex combinations to come in before we take the mixture.  

These show that goodness is not really a property of a type of distribution only, but also the specifics of how the class is assembled.

\section{A Class of Uniqueness Problems}

From now on we will use 
\begin{equation}\forall s\in \mathbb{R},\, \mathcal{\mathcal{M}}_\mu(s) =\int_{\mathbb{R}} e^{sx}d\mu(x)\end{equation} 
to mean the MGF (at $s\in \mathbb{R}$) of a probability measure $\mu$ on $\mathbb{R}$.  Here, $\mu$ may be random or deterministic.  We will often make assumptions about finiteness of the MGF, which we will state as needed. We also use
\begin{equation}\forall t\in \mathbb{R},\, \mathcal{\phi}_\mu(t) =\int_{\mathbb{R}} e^{itx}d\mu(x)\end{equation} 
to mean the characteristic function (at $t\in \mathbb{R}$) of a probability measure $\mu$ on $\mathbb{R}$.  Here, $\mu$ may be random or deterministic.  

In this section, we will discuss a variety of uniqueness problems regarding determining the distribution of a random probability measure from limited information.  These problems will all run parallel to classical versions of various uniqueness results.  
\begin{defn}
Given a probability measure $\mu$ defined on $\mathbb{R}$, let $\mu_k:=\int_{\mathbb{R}}x^k d\mu(x)$ denote the moment of order $k$ of $\mu$, when it exists.  Define $\mathfrak{C}$ to be the collection of $\mu \in P(\mathbb{R})$ that obey the Carleman condition that all the moments exist, are finite, and $\sum_{j=0}^\infty 1/\mu_{2j}^{1/2j}=\infty$ with the convention that $1/0=\infty$.  Let $M_{<\infty}$ denote the collection of $\mu \in P(\mathbb{R})$ with finite MGF in some neighborhood of $0$.\label{classes}\end{defn}
Sometimes $\mu$ will denote a random measure.

We will sometimes make the assumption that the random measure $\mu$ is uniformly bounded a priori. This means there exists $M>0$ such that $\mu$ is a.s. supported in $[-M, M]$.

We will use the notation $\mu_k:=\int_{\mathbb{R}}x^k d\mu(x)$ to indicate the (random) moment of order $k$ for $\mu$, wherever this is defined.

For any of the classes in Definition \ref{classes} or for $P^+$, we will say $\mu$ is a member of that class if this holds a.s.  Observe that all of these conditions are measurable.  Notice this implies no uniformity, for instance each sample from $\mu\in M_{<\infty}$ may correspond to a different open interval about which the MGF is finite.

For now, let us assume that $\mu$ is deterministic and state the classical uniqueness results for comparison:

\begin{prop}
Let $\mu, \nu$ be deterministic probability measures on $\mathbb{R}$.  Then $\mu=\nu$ provided any of the following hold:
\begin{enumerate}
\item $\phi_\mu=\phi_\nu$

\item $\mu, \nu\in P^+$ and $\mathcal{L}_\mu=\mathcal{L}_\nu$.

\item $\mu, \nu\in M_{<\infty}$ and $\mathcal{M}_\mu=\mathcal{M}_\nu$.

\item $\mu, \nu\in\mathfrak{C}$ and $\forall k \ge0,\, \mu_k=\nu_k$\end{enumerate}
\end{prop}

Let $S\phi, S\mathcal{L}, S\mathcal{M}$ denote respectively the collections of functions that arise from some $\mu$ as in case (1), (2), (3), with the value $+\infty$ possible in case (3).  Let $SMOM$ denote the collection of sequences of real numbers that arise as the moments of some $\mu$ as in case (4), which we regard as a function of $k\ge0$.  On each of these spaces of functions, we use the Borel $\sigma$ algebra generated by the evaluation maps.  Then the content of the last theorem is that $\mu \mapsto \phi_\mu$ is a bijection from $P(\mathbb{R})$ to $S\phi$, $\mu \mapsto \mathcal{L}_\mu$ is a bijection from $P^+$ to $S\mathcal{L}$, $\mu \mapsto \mathcal{M}_\mu$ is a bijection from the $M_{<\infty}$ to $S\mathcal{M}$, and $\mu\mapsto (\mu_k)_{k\ge0}$ is a bijection from $\mathfrak{C}$ to $SMOM$.  Observe that all of these maps are measurable isomorphisms.  It follows that

\begin{prop}
Let $\mu, \nu$ be random probability measures.  Then $\mu=_d\nu$ provided any of the following hold:
\begin{enumerate}
\item $(\phi_\mu(t))_{t\in \mathbb{R}}=_d(\phi_\nu(t))_{t\in\mathbb{R}}$

\item $\mu, \nu\in P^+$ and $(\mathcal{L}_\mu(s))_{s\ge0}=_d(\mathcal{L}_\nu(s))_{s\ge0}$.

\item $\mu, \nu\in M_{<\infty}$ and $(\mathcal{M}_\mu(s))_(s))_{s\in \mathbb{R}}=_d(\mathcal{M}_\nu(s))_{s\in \mathbb{R}}$ (which may be valued $\infty$ for some $s$ and some sample points.)

\item $\mu, \nu\in\mathfrak{C}$ and $(\mu_k)_{k\ge0}=_d(\nu_k)_{k\ge0}$\end{enumerate}
\end{prop}


To summarize, under suitable conditions, knowing the characteristic function, laplace transform, MGF, or moments jointly tells us the joint distribution of a random measure.  It is natural to ask what happens when this type of information is only known marginally. (We call these the marginal problems, as opposed to joint.)  Actually, we can already provide an answer to $2$ of these problems.  One counterexample that is uniformly bounded will simultaneously witness the failure of both statements.

\begin{thm}
There exist $\mu, \nu$ uniformly bounded random measures in $P^+$ such that $\mu\neq_d\nu$ while yet $\forall s \ge0,\, \mathcal{L}_\mu(s)=_d\mathcal{L}_\nu(s)$ and $\forall s \in \mathbb{R},\, \mathcal{M}_\mu(s)=_d\mathcal{M}_\nu(s)$.

\end{thm}

\begin{proof}
By Proposition \ref{example maker}, we obtain random variables $U, V\in T_4$ with $U\neq_dV$ such that $\forall y \in \mathbb{R},\, c_4(y)\cdot U=_dc_4(y)\cdot V$. Define random probability measures $\mu, \nu$ supported in $\{0, 1, 2, 3\}$ via $\mu=U_0\delta_0+U_1\delta_1+U_2\delta_2+U_3\delta_3$ and $\nu=V_0\delta_0+V_1\delta_1+V_2\delta_2+V_3\delta_3$.  Upon calculating the transforms of these random measures, the fact that $\forall y \in \mathbb{R},\, c_4(y)\cdot U=_dc_4(y)\cdot V$ translates into the fact that the (3') holds, and thus that (2') does as well.
\end{proof}

It is not difficult to believe that the moment problem will require a different argument to handle in the marginal case, but it may be surprising that the characteristic function marginal problem could not be handled in the above proof.  That is because the characteristic function is a complex valued random variable.  Indeed, we have been avoiding this situation partly because the methods of the last two sections cannot deal with complex valued uniqueness problems, and partly because we did not have to since the problem had already been solved in the negative, and special subclasses of interest already were nonnegative anyway.  However, here we arrive at a case where the signed question has not been solved yet, and our methods with some adjustment will actually be applicable.  We will present a proof of nonuniqueness for the marginal characteristic function case in Subsection \ref{margchf}.  From this proof, it will be plausible that the arithmetic dependences of the allowed support set is the culprit, because it reduces the analysis to one of polynomials, for which we will be able to use Lemma \ref{cramer}.

In Subsection \ref{margmom} we will present a proof of nonuniqueness in the moment marginal case, even if the restriction is made to a finite support set, hence to the uniformly bounded hypothesis.  In Subsection \ref{prime}, we will show how the relevant arithmetic structure from Subsection \ref{margmom} was multiplicative, and that therefore if the set of allowed values consists of coprime numbers, then uniqueness of the distribution of a random measure with a given set of moments does hold.

We display our results concerning random uniqueness problems in the following table. We remind the reader that the assumption for the MGF case is finite MGF in a neighborhood around $0$ (a.s.), the assumption for Laplace transform is nonnegativity, there are no assumptions for characteristic function, and the Carleman condition is assumed to hold (a.s.) for the moment problem.  The entries ``yes'' and ``no'' refer to whether or not uniqueness holds.  All answers ``no'' come with a uniformly bounded (pair of) counterexamples.

\begin{tabular}{l*{2}{c}}
                  & joint & marginal\\
\hline
MGF & yes & no \\
Nonnegative Laplace Transform           & yes & no \\
Characteristic Function           & yes & no \\
Moment Problem     & yes & no \\
\end{tabular}

\subsection{Uniqueness Fails for Characteristic Function Problem}

\label{margchf}

In this subsection, we will use $z=x+iy$ to denote a complex number.  We define $\forall n \ge0,\, P_n(x, y)=\Re(z^n)$ and $Q_n(x,y)=\Im(z^n)$ which are homogeneous polynomials in the two variables $x, y$ of degree $n$.  By default, our polynomials will be defined on a Euclidean space $\mathbb{R}^d$ where $d\ge1$ is the number of variables of the polynomial.  First we need a lemma.

\begin{lem}
For $N, l \in \mathbb{N}$ large enough, there exists $p$, a nonzero degree $l$ homogeneous polynomial in $N$ variables such that $\forall s, t, x, y \in \mathbb{R}$ we have 

\begin{equation}
p(sP_0(x, y)+tQ_0(x, y), \dots, sP_{N-1}(x, y)+tQ_{N-1}(x, y) )=0.\label{projective}
\end{equation}
\end{lem}

\begin{proof}

We first note that if some $N, l$ satisfies the conditions of the lemma, then  all greater pairs would work as well, so our use of the phrase ``large enough'' is justified.  We now only need to find one pair $N, l$ for which \eqref{projective} holds.

For any $l, N,$ let $S_{l, N}$ denote the real vector space of polynomials in $s, t, x, y$ of degree at most $lN$ and let $T_{l, N}$ denote the real vector space of degree $l$ homogeneous polynomials in $N$ (commuting) variables. 

Observe that

\begin{equation}p(sP_0(x, y)+tQ_0(x, y), \dots, sP_{N-1}(x, y)+tQ_{N-1}(x, y) )\end{equation}
defines a polynomial (not necessarily homogeneous) of degree at most $lN$ in the $4$ variables $s, t, x, y$.  Define the corresponding evaluation map

\begin{equation}\Phi: T_{l, N}\rightarrow S_{l, N}\end{equation}
via

\begin{equation}\Phi(p)=p(sP_0(x, y)+tQ_0(x, y), \dots, sP_{N-1}(x, y)+tQ_{N-1}(x, y) ).\end{equation}
Observe that $\Phi$ is linear.  The dimension of its codomain is $\sum_{j=0}^{lN} \binom{j+3}{3}\le (lN+4)^4.$ If we restrict to $l=N$ then the dimension of the domain is $\binom{2N-1}{N}$.

Therefore, for $N=l$ large enough, the kernel of $\Phi$ is nontrivial.

\end{proof}

As before, we will construct a precursor to the counterexample by constructing merely bounded random vectors with the desired property, and then we will adjust them so as to turn them into random probability vectors.  The following can be thought of as the complex analog of a part of the argument used in the proof of Proposition \ref{example maker}.  The equality in distributions are meant for complex-valued random variables.

\begin{lem}

For $N$ large enough, there exist 

\begin{equation}U=(U_0, \dots, U_{N-1}), V=(V_0, \dots, V_{N-1})\end{equation}
bounded random vectors taking values in $\mathbb{R}^N$ for which $\forall z\in\mathbb{C}$ we have

\begin{equation}
\sum_{j=0}^{N-1}z^jU_j=_d\sum_{j=0}^{N-1}z^jV_j \label{eqdistr}
\end{equation}
while yet $U\neq_dV$

\end{lem}

\begin{proof}

Take $N$ large enough so that the set 

\begin{equation}\{(sP_0(x, y)+tQ_0(x, y), \dots, sP_{N-1}(x, y)+tQ_{N-1}(x, y))\in\mathbb{R}^N|s, t, x, y\in \mathbb{R}\}\end{equation}
is contained in a projective variety.  By Lemma \ref{cramer} we may find $U, V$ bounded random vectors so that $\forall s, t, x, y $ we have

\begin{align*}(sP_0(x, y)+tQ_0(x, y), \dots, sP_{N-1}(x, y)+tQ_{N-1}(x, y))\cdot U=_d\\(sP_0(x, y)+tQ_0(x, y), \dots, sP_{N-1}(x, y)+tQ_{N-1}(x, y))\cdot V\numberthis\end{align*}

while yet $U\neq_d V$.

By the Cram\'er Wold device, we have that $\forall x, y \in \mathbb{R}$, the $\mathbb{R}^2$-valued random vectors

\begin{equation}\big((P_0(x, y), \dots, P_{N-1}(x, y))\cdot U, (Q_0(x, y), \dots, Q_{N-1}(x, y))\cdot U\big)\end{equation}
and

\begin{equation}\big((P_0(x, y), \dots, P_{N-1}(x, y))\cdot V, (Q_0(x, y), \dots, Q_{N-1}(x, y))\cdot V\big)\end{equation}
have the same distribution.  But by real isomorphism of $\mathbb{R}^2$ with $\mathbb{C}$, this is the same as saying that

\begin{align*}
(P_0(x, y), \dots, P_{N-1}(x, y))\cdot U + i((Q_0(x, y), \dots, Q_{N-1}(x-y))\cdot U)=_d\\(P_0(x, y), \dots, P_{N-1}(x, y))\cdot V + i((Q_0(x, y), \dots, Q_{N-1}(x, y))\cdot V) \numberthis
\end{align*}
But by the definition of the $P_n, Q_n$ this is the same as saying $\forall z \in \mathbb{C}$ we have

\begin{equation}(z^0, \dots, z^{N-1})\cdot U=_d (z^0, \dots, z^{N-1})\cdot V\end{equation}
as complex-valued random variables.  Thus, $U$ and $V$ have the required properties.
\end{proof}

\begin{lem}
There is $N'$ large enough so that there exist 

\begin{equation}U'=(U'_0, \dots, U'_{N-1})\neq_dV'=(V'_0, \dots, V'_{N-1}),\end{equation}
both valued in 

\begin{equation}T_{N'}:=\{(x_0, \dots, x_{N'-1})|\forall j:0\le j\le N'-1,\, x_j\ge0, \sum_{j=0}^{N'-1}x_j=1\},\end{equation}
for which $\forall z \in \mathbb{C}$

\begin{equation}\sum_{j=0}^{N'-1}z^jU'_j=_d\sum_{j=0}^{N'-1}z^jV'_j\label{final}.\end{equation}
\end{lem}

\begin{proof}

Take $N, U, V$ from the last lemma.  Set $N'=N+1$.  Notice that \eqref{eqdistr} still holds if we rescale or translate $U, V$ in the same way.  Consider $H:=\{(x_0, \dots, x_{N-1})|1\ge x_0\ge\dots\ge x_{N-1}\ge 0\}$.  Since $H$  has nonempty interior, we may take $U, V \in H$ without loss of generality.  Then when we define $U_{-1}=1=V_{-1}, U_N=0=V_N,$ and $\forall j: 0\le j\le N,\, U'_j=U_{j-1}-U_{j}, V'_j=V_{j-1}-V_{j}$, observe that $U', V'$ are $T_{N'}$ valued.  Furthermore, $U'\neq_d V'$.  A calculation shows that we have $\forall z \in \mathbb{C}$

\begin{align*}
(z^0, \dots, z^{N-1})\cdot (U'_1+\dots+U'_{N}, U'_2+\dots+U'_{N}, \dots, U'_{N})=_d\\(z^0, \dots, z^{N-1})\cdot (V'_1+\dots+V'_{N}, V'_2+\dots+V'_{N}, \dots, V'_{N}).\numberthis
\end{align*}

From this it follows that

\begin{align*}
(z^0)U'_1+(z^0+z^1)U'_2+\dots+(z^0+\dots+z^{N-1})U'_N=_d \\(z^0)V'_1+(z^0+z^1)V'_2+\dots+(z^0+\dots+z^{N-1})V'_N.\numberthis
\end{align*}

By multiplying by $(z-1)$, using the definition of $T_{N'}$, and adding $1$ to both sides, we obtain $\forall z \in \mathbb{C}$

\begin{equation}\sum_{j=0}^{N'-1}z^jU'_j=_d\sum_{j=0}^{N'-1}z^jV'_j\end{equation}

Thus, $U'$ and $V'$ have the required properties.
\end{proof}

We use the homeomorphism, hence measurable isomorphism of $T_{N'}$ with the collection of probability measures supported in $\{0, \dots, N'-1\}$ given by $v\mapsto v_0\delta_0+\dots+v_{N'-1}\delta_{N'-1}$ in order to see the following:

\begin{thm}
For $N'$ large enough, there are two random probability measures $\mu, \nu$ with distinct distributions which are a.s. supported in $\{0, \dots, N'-1\}$ for which $\forall t \in \mathbb{R},\, \phi_\mu(t)=_d\phi_\nu(t)$.

\end{thm}

\begin{proof}

Obtain $U', V', N'$ as in the last lemma, and set $\mu=U'_0\delta_0+\dots+U'_{N'-1}\delta_{N'-1}$ and $\nu=V'_0\delta_0+\dots+V'_{N'-1}\delta_{N'-1}$.  Set $z=e^{it}$ in \eqref{final} so that its left side is the random characteristic function of $\mu$ and its right side is the random characteristic function of $\nu$.  This construction has the required properties.

\end{proof}

\subsection{Too Many Multiplicative Relationships Spoils Uniqueness of the Marginal Moment Problem}

\label{margmom}

Let us see that $4$ values is again enough to prove nonuniqueness in the marginal version of the random moment problem:

\begin{thm}

There exist $\mu, \nu$ random measures a.s. supported in $\{1, 2, 4, 8\}$ for which $\forall k \ge0,\, \mu_k=_d\nu_k$ while yet $\mu\neq_d\nu$.

\end{thm}

\begin{proof}

By the identification of $T_4$ with the collection of probability measures supported in $\{1, 2, 4, 8\}$ (given by $v\mapsto v_0\delta_1+v_1\delta_2+v_2\delta_4+v_3\delta_8$) we have that it suffices to find two random probability vectors $U=(U_0, U_1, U_2, U_3), V=(V_0, V_1, V_2, V_3)$ for which $U\neq_d V$ and $\forall k \ge0$

\begin{equation}
\forall k\ge0,\, U_0*1^k+U_12^k+U_24^k+U_38^k=_dV_0*1^k+V_12^k+V_24^k+V_38^k.\label{moment}
\end{equation}
That this suffices is because upon defining $\mu=U_0\delta_1+U_1\delta_2+U_2\delta_4+U_3\delta_8$ and $\nu=V_0\delta_1+V_1\delta_2+V_2\delta_4+V_3\delta_8$, which are random probability measures a.s. supported on $\{1, 2, 4, 8\}$, we would have that the left side of \eqref{moment} is $\mu_k$ and the right side is $\nu_k$.

We may write \eqref{moment} as

\begin{equation}\forall k\ge0,\, U_0*2^{0k}+U_12^{1k}+U_22^{2k}+U_32^{3k}=_dV_0*2^{0k}+V_12^{1k}+V_22^{2k}+V_32^{3k}.\end{equation}
Pick $U, V$ from Lemma \ref{example maker}.  

\end{proof}

\subsection{How The Right Kind of Arithmetic Independence Can Help}

\label{prime}

Let $a_0, \dots, a_{N-1}$ be a list of nonzero natural numbers which are pairwise coprime.

\begin{thm}

If $\mu, \nu$ are random probability measures which are a.s. supported in $\{a_0,\dots,a_{N-1}\}$ and $\forall k\ge0,\, \mu_k=_d\nu_k$ then $\mu=_d\nu$.

\end{thm}

\begin{proof}

First we consider the homeomorphism between the space of probability measures supported on $\{a_0, \dots, a_{N-1}\}$ with $T_{N}:=\{(x_0, \dots, x_{N-1})|\forall j:0\le j\le N-1,\, x_j\ge0, \sum_{j=0}^{N-1}x_j=1\}$ given by $\theta \mapsto (\theta(\{a_0\}),\dots,\theta(\{a_{N-1}\}))$.  It suffices to show that if $U, V$ are $T_N$ valued and $\forall k \ge0$ we have

\begin{equation}\sum_{j=0}^{N-1}(a_j)^kU_j=_d \sum_{j=0}^{N-1}(a_j)^kV_j\label{last}\end{equation}
then $U=_dV$.  Once this is shown, if $\mu, \nu$ are a.s. supported on $\{a_0, \dots, a_{N-1}\}$, then upon defining \begin{equation}U=\bigg(U_0,\dots, U_{N-1}\bigg)=\bigg(\mu(\{a_0\}), \dots, \mu(\{a_{N-1}\})\bigg)\end{equation}
and
\begin{equation}V=\bigg(V_0, \dots, V_{N-1}\bigg)=\bigg(\nu(\{a_0\}), \dots, \nu(\{a_{N-1}\})\bigg)\end{equation}
we would have that the left side of \eqref{last} is $\mu_k$ and the right side is $\nu_k$.

By Lemma \ref{cramer}, it suffices to show that $\{(a_0^k, \dots, a_{N-1}^k)|k\ge0\}$ is not contained in any projective variety.  

Suppose to the contrary that there exists some nonzero homogeneous polynomial $p$ in $N$ variables of degree $l>0$ for which $\forall k\ge0,\, p(a_0^k, \dots, a_{N-1}^k)=0$.  Enumerate (without repetitions) the monomials $\{M_i\}_{j=1}^K$ of degree $l$, and write $p=\sum_{j=1}^K b_jM_j$ where $b_j\in \mathbb{R}$.  For each $j,\, \exists C_j\in \mathbb{N}$ such that 

\begin{equation}\forall k \ge0,\, M_j(a_0^k, \dots, a_{N-1}^k)=(C_j)^k.\end{equation}  
Because the $M_j$ are distinct, and $a_0, \dots, a_{N-1}$ are all coprime, we find that all the $C_j$ are distinct.  Find $j_0$ such that $C_j$ is the largest, subject to the constraint that $b_j\neq0$.  This is a feasible optimization problem because $p\neq 0$.  There will be only one optimal solution because the $C_j$ are all distinct, and a solution exists because the list of $C_j$ is finite.  Consider $\lim_k p(a_0^k, \dots, a_{N-1}^k)/C_{j_0}^k=b_{j_0}\neq0$ for a contradiction with the fact that $\forall k \ge0,\, p(a_0^k, \dots, a_{N-1}^k)/C_{j_0}^k=0$.  This completes the proof.

\end{proof}

\section{Acknowledgements}

The author would like to thank his adviser, Marek Biskup, for his part in conversations concerning the present project.  The author would also like to thank his oral exam committee, especially Tom Liggett, for helpful suggestions.  The author is grateful to Ji\v r\'i \v Cern\'y for suggesting some of the problems in this article.  Finally, the author is thankful for funding for this project coming from the NSF award DMS-1407558.

\end{document}